\documentclass[11pt]{article}

\usepackage{amsmath,amsthm,amssymb}
\usepackage{algorithm}
\usepackage{algpseudocode}
\usepackage[dvipsnames]{xcolor}
\usepackage{hyperref}
\usepackage[utf8]{inputenc}
\theoremstyle{plain} 
\newtheorem{theorem}{Theorem}
\newtheorem{lemma}{Lemma}
\newtheorem{corollary}{Corollary}
\newtheorem{heuristic}{Heuristic}
\theoremstyle{definition}
\newtheorem*{definition}{Definition} 
\newtheorem{exmp}{Example}[section]
\theoremstyle{remark}

\usepackage[headheight=14pt, margin=1in]{geometry}

\usepackage[title]{appendix}

\usepackage{fancyhdr}
\newcommand{\ex}{\mathbb{E}}

\title{Distribution of Nonzero Digits in a Greedy Sequence of Powers of Two}
\date{}
\author{David Wu}
\begin{document}
\maketitle
\begin{abstract}
Understanding the distribution of digits in the expansions of perfect powers in different bases is difficult. Rather than consider the asymptotic digit distributions, we consider the base-10 digits of a restricted sequence of powers of two. We apply elementary methods to show that this sequence of powers of two can be constructed to preserve trailing digits while locally maximizing the number of zeros between nonzero digits. We also provide a heuristic description of the trailing digits of these powers of two. 
\end{abstract}

\section{Introduction}

Let's consider the distribution of the base-10 digits of powers of two. A natural heuristic argues that the fraction of nonzero digits should be $9/10$. Previous results proved bounds on the density of nonzero digits, such as Stewart's lower bound of $\ln(x)/(C + \ln\ln(x))$ nonzero digits \cite{CS} or Radcliffe's elementary lower bound of $\log_4(x)$ nonzero digits \cite{DR}. Instead of proving a lower bound over all powers of two, we consider the problem of locally maximizing the sparsity of nonzero digits in a 10-adic sequence of powers of $2$. We use a very greedy algorithm that has surprising properties, and is well-suited to analyzing 10-adic powers of $2$. 

Recall that every power of two (and every natural number) can be uniquely represented in the form \[2^{\chi} = b_1 \cdot 10^{d_1} + b_2 \cdot 10^{d_2} + \cdots, \] where $d_1 < d_2 < \cdots$ and $1 \le b_i \le 9$. In other words, this representation singles out the exponents $d_n$ that correspond to nonzero digits in the base-10 expansion of $2^\chi$. Our goal is construct a sequence of exponents $p_n$ of $2$ such that the the sequence $d_1, d_2, \cdots$ is locally sparse. That is, we try to locally maximize $d_n - d_{n-1}$, the gaps between adjacent nonzero digits. We explain the algorithm formally and provide a running example.

We proceed as follows. Suppose we have constructed $p_n$ satisfying \[2^{p_n} \equiv b_1 \cdot 10^{d_1} + b_2 \cdot 10^{d_2} + \cdots + b_n \cdot 10^{d_n} \pmod{10^{d_n+1}},\] where $1 \le b_i \le 9$ and $0 = d_1 < d_2 < \cdots$. 

Then, for each $p_{n+1}$, repeat this process to find $p_{n+2}$. If some value of $p_{n+1}$ yields a value of $p_{n+2}$ that is smaller than the value of $p_{n+2}$ by some other $p_{n+1}$, discard the smaller $p_{n+1}$ and its descendants $p_{n+2}$, etc.

Consider all powers $2^p$ with \[2^p \equiv b_1 \cdot 10^{d_1} + b_2 \cdot 10^{d_2} + \cdots + b_n \cdot 10^{d_n} \pmod{10^{e}},\] for $e > d_n$ and retain only those exponents $p$ that give the largest value of $e-d_n$. We denote each of these values of $p$ as $p_{n+1}$. For example, let $p_1 = 3$ so $2^{p_1} \equiv 8 \cdot 10^0 \pmod{10^1}$. In our example, we look at $p > 3$ and find that $p \equiv 3 \pmod{100}$ gives $e - d_1 = 3$. We can show that this is the largest possible $e-d_1$. Hence, we can take $p_2$ to be the residue class $3 \pmod{100}$. 

In our example, let's take $p_2 = 103$. Then \[2^{p_2} = 8 \cdot 10^0 + 3 \cdot 10^3 + \cdots .\] Then we look for $p > 103$ such that $e-d_2$ (here, we can see that $d_2 = 3$ is maximal). We find that if $p \equiv 2403 \pmod{12500}$, then $e - d_2 = 2$. However, if $p \equiv 2013 \pmod{12500}$ then $e - d_2 = 3$, and it can be shown that this is the maximum possible $e - d_2$. Therefore, we take $p_3$ to be the residue class $2013 \pmod{12500}$ and discard all other residue classes such as $p \equiv 2403 \pmod{12500}$. 

We illustrate the next two iterations. Taking $p_3 = 2103$, we find \[2^{p_3} = 8 \cdot 10^0 + 3 \cdot 10^3 + 1 \cdot 10^6 + \cdots \] Then we search for $p > 2103$ and find $p_4 = 670414603$. This gives \[2^{p_4} = 8 \cdot 10^0 + 3 \cdot 10^3 + 9 \cdot 10^6 + 1 \cdot 10^{13} + \cdots.\] Hence we only preserve the residue class $670414603 \pmod{4 \cdot 5^{12}}$. The process continues thereafter.

This yields a sequence of exponents $p_1, p_2, p_3, \ldots$ of $2$ and corresponding exponents $d_1, d_2, d_3, \ldots$ of $10$. In other words, we are using a greedy strategy to construct a $10$-adic power of $2$ whose infinite $10$-adic expansion has many zeros. 

In other words, when constructing $p_{n+1}$, we preserve $b_i$ and $d_i$ for $1 \le i \le n$. Suppose we have constructed $p_n$ satisfying \[2^{p_n} \equiv b_1 \cdot 10^{d_1} + b_2 \cdot 10^{d_2} + \cdots + b_n \cdot 10^{d_n} \pmod{10^{d_n+1}},\] where all $b_i \neq 0$ and $0 = d_1 < d_2 < \cdots$. We only consider powers $2^p$ with \[2^p \equiv b_1 \cdot 10^{d_1} + b_2 \cdot 10^{d_2} + \cdots + b_n \cdot 10^{d_n} \pmod{10^{e}}.\]

With this restriction, we show in Section \ref{sec:3} that there are very few possible values of $p$. More specifically, suppose we have constructed $p_n \pmod{4 \cdot 5^{d_k-1}}$. Then if $e = d_n+1$, the only valid exponents $p$ are $p \equiv p_n + i \cdot 4 \cdot 5^{d_n-1} \pmod{4 \cdot 5^{e}}$ for $1 \le i \le 5$. Pseudocode for the algorithm can be found in the Appendix.

Our strategy does not necessarily give powers of $2$ with the highest density of zeros. For example, $2^{103}$ has only $5$ zeros among its 32 digit, while $2^{102}$ has seven zeros among its $31$ digits. However, we obtain many zeros in the early terms of the expansion. 

In Section \ref{sec:3}, we show that we always have $d_{n+1} - d_n \ge 2$, so we can always obtain at least $2$ consecutive zeros after each nonzero digit. Moreover, we show that our constraints create a unique sequence of $d_n$. We also investigate various other phenomena with digit distributions that occur. In Section \ref{sec:4}, we provide a heuristic that predicts that the density $\#\{d_n \le x\}/x \rightarrow 13/4$ as $x \rightarrow \infty$. This agrees well with the data we computed in Tables 1 and 2.

\section{Properties of the greedy sequence} \label{sec:3}
We first state and prove a lemma. 
\begin{lemma}\label{lemma:1}Let $i \ge j \ge m \ge 1$. Then $2^i \equiv 2^j \pmod{10^m}$ if and only if $i \equiv j \pmod{4 \cdot 5^{m-1}}$. 
\end{lemma}
\begin{proof}
Since $2$ is a primitive root mod $5^n$, we have $2^i \equiv 2^j \pmod{5^m}$ if and only if $i \equiv j \pmod{4 \cdot 5^{m-1}}$. Now since $i \ge j \ge m$, we have $2^i \equiv 2^j \equiv 0 \pmod{2^m}$. The result follows from Chinese Remainder Theorem.
\end{proof}

\begin{lemma}\label{lemma:2}
Let $i > m \ge 1$. Suppose $2^i \equiv x \pmod{10^m}$. Write \[2^{i + 4j5^{m-1}} \equiv x + a_j10^{m} \pmod{10^{m+1}},\] with $0 \le a_j \le 9$. As $j \ge 0$ runs through the integers $\pmod{5}$, the values of $a_j$ run through all the integers in either $\{1, 3, 5, 7, 9\}$ or $\{0, 2, 4, 6, 8\}$. 
\end{lemma}

\begin{proof}Since $i > m$, we know that $2^{i + 4j5^{m-1}} \equiv 2^i \equiv 0 \pmod{2^{m+1}}$. Since $2^m$ divides $x$, we can divide through by $2^m$, which shows that $x + a_j \equiv 0 \pmod{2}$ Therefore the parity of $a_j$ is determined by $x$. Lemma \ref{lemma:1} says that distinct values of $j \pmod{5}$ give distinct values of $a_j$. The lemma follows easily. 
\end{proof}

For $n = 1, 2, 3, \ldots ,$ choose integers $p_n$ and write
\[2^{p_n} \equiv b_1 + b_210^{d_2} + b_310^{d_3} + \cdots +  b_n10^{d_n} \pmod{10^{d_n+1}},\]
where $1 \le b_i \le 9$ for each $i$ and where $d_1 = 0 < d_2 < d_3 < \cdots < d_n$. Assume that the $p_n$ are chosen by the ``greedy strategy''. That is, $p_1, p_2, \ldots, p_{n-1}$ are
chosen so that $p_n > d_n$ and $d_n - d_{n-1}$ is as large as possible.

\begin{theorem}\label{thm:1}If the $p_n$ are chosen by the greedy strategy, then the $b_i$ are odd for $i \ge 2$. 
\end{theorem}
\begin{proof}
In Lemma \ref{lemma:2}, let $x = b_1 + \cdots + b_{n-1} 10^{d_{n-1}}$, let $i = p_{n-1}$, and let $m = d_n$. Write \[2^{p_{n-1} + 4j5^{d_n}} \equiv x + a_j10^{d_n} \pmod{10^{d_n+1}}.\] For some $j = j_0$, we have $a_{j_0} = b_n$. Lemma \ref{lemma:1} shows that as $j$ runs through the integers $\pmod{5}$, $a_j$ will run through the integers between $0$ and $9$ that have the same parity as $b_n$. Therefore, if $b_n$ is even, we can pick $j$ such that $a_j = 0$. This means that $d_n$ is not maximal, so $b_n$ must be odd, as desired.  
\end{proof}

\begin{theorem}\label{thm:2}
Let \[2^{p_{n+1}} \equiv b_1 \cdot 10^{d_1} + b_2 \cdot 10^{d_2} + \cdots + b_{n+1} \cdot 10^{d_{n+1}} \pmod{10^{d_{n+1}+1}},\] where $p_n$ is chosen by the greedy strategy described in the introduction. Then $d_{n+1} - d_n \ge 2$. 
\end{theorem}
\begin{proof}
Suppose $d_{n+1} - d_n = 1$. Then \[2^{p_n + 4j \cdot 5^{d_n-1} + 4k \cdot 5^{d_n}} \equiv b_1 \cdot 10^{d_1} + b_2 \cdot 10^{d_2} + \cdots + a_j \cdot 10^{d_n} + c_{j,k} \cdot 10^{d_n+1} \pmod{10^{d_n+2}},\] with $1 \le a_j, c_{j,k} \le 9$ for $j,k = 0,1,2,3,4$. Theorem \ref{thm:1} says that $a_j$ and $c_{j,k}$ must be odd. When $j \ge 1$, \[\left(2^{4(j-1)5^{d_n-1} + 4k \cdot 5^{d_n}} - 1\right)2^{p_n + 4 \cdot 5^{d_n-1}} \equiv (a_j-a_1) \cdot 10^{d_n} + (c_{j,k} - c_{1,0}) \cdot 10^{d_n+1} \pmod{10^{d_n+2}}.\] Since $5^{d_n}$ divides \[2^{4(j-1)5^{d_n-1} + 4k \cdot 5^{d_n}} - 1\] and $p_n + 4 \cdot 5^{d_n-1} \ge d_n + 2$ (this is why we used $j \ge 1$ instead of $j \ge 0$), we can divide by $10^{d_n}$ to obtain \begin{equation}\label{eq:mod4}
2^2N \equiv (a_j - a_1) + 10(c_{j,k} - c_{1,0}) \pmod{100},
\end{equation} for some integer $N$. Since $c_{j,k}$ and $c_{1,0}$ are odd, we find that $a_j \equiv a_1 \pmod{4}$. This allows at most three values of $a_j$. But Lemma \ref{lemma:2} says that $a_j$ runs through all values congruent to $a_1 \pmod{2}$, and we have omitted only $j = 0$ from consideration, so there are $4$ remaining distinct values of $a_j$. This contradiction implies that $d_{n+1} - d_n \ge 2$.  
\end{proof}

We now illustrate an example of Lemma \ref{lemma:2}.
\begin{exmp}\label{examp3}
We can compute $2^{103} = 8 + 3 \cdot 10^3 + 4 \cdot 10^4 + \cdots$. Here the third nonzero digit is $4$, which is even. As Lemma \ref{lemma:2} claims, we can construct the modified exponent $p' = 103 + 4 \cdot \varphi(5^4)$ such that $2^{p'} = 8 + 3 \cdot 10^3 + 1 \cdot 10^6 + \cdots$. The effect of adding $4 \cdot \varphi(5^4)$ to the exponent is that the nonzero even digit $4$ corresponding to $10^4$ was changed to a $0$. Furthermore, note that aside from the units digit, all of the nonzero digits in $2^{p'} \pmod{10^7}$ are odd.
\end{exmp}

Now suppose that we have greedily constructed $p_1, p_2, \ldots, p_{n}$. Recall that \[2^{p_n} = b_1 \cdot 10^{d_1} + b_2 \cdot 10^{d_2} + \cdots + b_n \cdot 10^{d_n} + \cdots.\] We claim that there is a unique choice of $b_n$ that corresponds with $p_{n+1}$.

\begin{definition}
For $1 \le k \le 9$, We say that $b_n=k$ is \textit{$\beta$-forceable} if there exists $e$ such that $2^{p_n} \equiv 2^e \pmod{10^{d_n+\beta}}$. In other words, we can find $e$ such that $2^e$ has at least $\beta$ zeros between $b_n$ and $b_{n+1}$. Equivalently, we require that $d_{n+1} - d_n \ge \beta + 1$. On the other hand, we say that $b_n$ is \textit{$\beta$--unforceable} if there does not exist $e$ such that $2^{p_n} \equiv 2^e \pmod{10^{d_n+\beta}}$.
\end{definition}

\begin{exmp}\label{examp:enforc}
We consider $p_1 = 3, p_2 = 103$, with $d_1 = 0$ and $d_2 = 3$. So $2^{p_1} = 8 \cdot 10^0$ and $2^{p_2} = 8 \cdot 10^0 + 3 \cdot 10^3 + \cdots$. Consider the values of $2^{p_2 + 4k \cdot 5^2}$:  \begin{align*}
    2^{p_2 + 0 \cdot 5^2} &= 8 \cdot 10^0 + 3 \cdot 10^3 + 4 \cdot 10^4 + 6 \cdot 10^5 + \cdots \\
    2^{p_2 + 4 \cdot 5^2} &= 8 \cdot 10^0 + 1 \cdot 10^3 + 1 \cdot 10^4 + 4 \cdot 10^5 + \cdots\\
    2^{p_2 + 8 \cdot 5^2} &= 8 \cdot 10^0 + 9 \cdot 10^3 + 7 \cdot 10^4 + 1 \cdot 10^5 + \cdots\\
    2^{p_2 + 12 \cdot 5^2} &= 8 \cdot 10^0 + 7 \cdot 10^3 + 4 \cdot 10^4 + 9 \cdot 10^5 + \cdots\\
    2^{p_2 + 16 \cdot 5^2} &= 8 \cdot 10^0 + 5 \cdot 10^3 + 1 \cdot 10^4 + 7 \cdot 10^5 + \cdots
\end{align*} Therefore, by looking at the parity of $b_3$ for each possible $b_2$, we see that $b_2 = 1$, $b_2 = 5$, and $b_2 = 9$ are $0$-\emph{forceable} (equivalently, they are $1$-\emph{unforceable}). Therefore, we only need to consider $b_2 = 3$ and $b_2 = 7$.

Next, we calculate $2^{p_2+0\cdot 5^2 + 4k \cdot 5^3}$ and $2^{p_2+12 \cdot 5^2 + 4k \cdot 5^3}$ for $0 \le k \le 4$ and find
\begin{align*}
    2^{p_2 + 0 \cdot 5^2 + 16 \cdot 5^3} &= 8 \cdot 10^0 + 3 \cdot 10^3 + 1 \cdot 10^6 + \cdots \\
    2^{p_2 + 12 \cdot 5^2 + 16 \cdot 5^3} &= 8 \cdot 10^0 + 7 \cdot 10^3 + 3 \cdot 10^5 + \cdots
\end{align*}
Note that there is a unique choice of $k$ that forces a zero after $b_2$. Therefore, we see that $b_2 = 3$ is $2$-\emph{forceable}, while $b_2 = 7$ is $1$-\emph{forceable}.
\end{exmp}

\begin{corollary}\label{cor:1} Consider the sequence $p_1$, $p_2$, $\cdots$, $p_{n-1}$, $p_n$. Suppose $b_n$ is $2$-\emph{forceable}. If $b_n = 3$ then it is the unique $2$-\emph{forceable} $b_n$. The same result holds if $b_n = 7$. That is, there is no alternate sequence $p_1$, $p_2$, $\cdots$, $p_{n-1}$, $p_n'$ such that $b_n' \neq b_n$ and $b_n'$ is $2$-\emph{forceable}. 
\end{corollary}

\begin{proof}
Suppose that $b_n$ is $2$-\emph{forceable} and $b_n = 3$; the proof for $b_n = 7$ is analogous. By definition, we can construct $p_{n}$ such that $2^{p_{n}} \equiv x  + 3 \cdot 10^{d_n}  \pmod{10^{d_n+3}}$, where $x$ is a $d_n$ digit integer. We split the proof into two cases.

\vspace{0.5cm}
\noindent\emph{Case 1:} We replace $b_n$ with $1$, $5$, or $9$. 

This case is simple; working modulo $4$ as in the proof of Theorem \ref{thm:2}, we must have $b_n - 3 \equiv 0 \pmod{4}$. So $b_n$ cannot be $1$, $5$, or $9$.

\noindent\emph{Case 2:} We replace $b_n$ with $7$.

If we replace $b_n$ with $7$, we can force at least one more zero by using the modulo $4$ argument as in the proof of Theorem \ref{thm:2}. Now we have \begin{align*}
    2^{p_{n}} &\equiv x + 3 \cdot 10^{d_n} \pmod{10^{d_n+3}} \\
    2^{p_{n}'} &\equiv x + 7 \cdot 10^{d_n} + a \cdot 10^{d_n+2} \pmod{10^{d_n+3}},
\end{align*}
where $0 \le a \le 9$ and $p_n' = p_n + 4k \cdot 5^{d_n-1}$ for some $1 \le k \le 4$. Now we have $p_n \ge d_n$ and $5^{d_n} \mid 2^{p_n' - p_n} - 1 = 2^{4k \cdot 5^{d_n-1}} - 1$. Therefore, we can divide out by $10^{d_n}$; we arrive at $4 + 100a \equiv 0 \pmod{1000}$. Thus $a$ must be odd, which implies that $b_n = 7$ is $2$-\emph{unforceable}, as desired.
\end{proof}

\begin{corollary}\label{cor:2} Suppose that $b_n$ is $2$-\emph{forceable}. If $b_n=5$, then it is the unique $2$-\emph{forceable} $b_n$. 
\end{corollary}
\begin{proof}
The proof is entirely analogous to the proof of Corollary \ref{cor:1}.
\end{proof}

The same property does not hold if either $1$ or $9$ happen to be $2$-\emph{forceable} because their difference is a multiple of $8$. However, the following two properties do hold.

\begin{corollary}\label{cor:3} $b_n=1$ is $2$-\emph{forceable} if and only if $b_n=9$ is $2$-\emph{forceable}. 
\end{corollary}
\begin{proof}
Suppose that $p_1$, $p_2$, $\ldots$, $p_n$ has been constructed such that $b_n = 1$. Then the modified sequence $p_1$, $p_2$, $\ldots$, $p_n'$ has $b_n' = 9$. 

Applying the modulo $4$ argument as in Theorem \ref{thm:2}, we find that if $b_n=1$ is $1$-\emph{forceable}, then $b_n' = 9$ is also $1$-\emph{forceable}. The reverse implication also follows. Now assume that $b_n=1$ is $1$-\emph{forceable}. So we have \begin{align*}
    2^{p_{n}} &\equiv x + 1 \cdot 10^{d_n} + a \cdot 10^{d_n+2} \pmod{10^{d_n+3}} \\
    2^{p_{n}'} &\equiv x + 9 \cdot 10^{d_n} + b \cdot 10^{d_n+2} \pmod{10^{d_n+3}},
\end{align*}
for $0 \le a, b \le 9$ and $p_n' = p_n + 4k \cdot 5^{d_n-1}$ for $1 \le k \le 4$. 

Subtracting and dividing by $10^{d_n}$, we obtain $8 + 100(a-b) \equiv 0 \pmod{1000}$. Therefore, $a \equiv b \pmod{2}$. By hypothesis, $b_n = 1$ is $2$-\emph{forceable}, so $a$ is even. Therefore $b$ is also even, which means that $b_n = 9$ is also $2$-\emph{forceable}, as desired.
\end{proof}

\begin{corollary}\label{cor:4} Suppose that $b_n$ is $3$-\emph{forceable}. If $b_n = 1$, then it is the unique $3$-\emph{forceable} $b_n$. The same result holds if $b_n=9$. That is, exactly one of $b_n=1$ and $b_n=9$ is $3$-\emph{forceable}; the other 
$b_n$ is $3$-\emph{unforceable}. \end{corollary}
\begin{proof}
Suppose that $b_n=1$ or $b_n=9$ is $2$-\emph{forceable}. Then by Corollary \ref{cor:3}, there exist $p_n$ and $p_n'$ such that 

\begin{align*}
    2^{p_{n}} &\equiv x + 1 \cdot 10^{d_n} + a \cdot 10^{d_n+3} \pmod{10^{d_n+4}} \\
    2^{p_{n}'} &\equiv x + 9 \cdot 10^{d_n} + b \cdot 10^{d_n+3} \pmod{10^{d_n+4}},
\end{align*}
where again $0 \le a, b \le 9$ and $p_n' = p_n + 4k \cdot 5^{d_n-1}$ for $1 \le k \le 4$. 

Subtracting and dividing out by $10^{d_n}$, we have $8 + 1000(a-b) \equiv 0 \pmod{10000}$. So $a$ and $b$ must have different parities. If $b_n = 1$ is $3$-\emph{forceable}, then $a$ is even, so $b$ is odd. Therefore $b_n = 9$ is $3$-\emph{unforceable}, as desired. The same logic holds if $b_n = 9$ is $3$-\emph{forceable}. The claim is thus proved. 
\end{proof}

\section{Heuristic argument for the expected sparsity of nonzero digits}\label{sec:4}
We justify a heuristic that predicts the expected number of zeros between nonzero digits in our sequence. 

\begin{heuristic}[Expected Gap of zeros] Let $\ex[x]$ be the expected value of $x$. Then \[\ex[\text{\# of zeros between nonzero digits}] = \dfrac{13}{4}\] \end{heuristic}

\begin{proof}
Empirical evidence (refer to Table \ref{freq}) suggests that the probability that a nonzero digit is 1 or 9 is approximately $\dfrac{1}{4}$. We can also think of $1$ and $9$ being their own group of digits, with $3, 5,$ and $7$ being the other three groups of digits. According to Corollaries 1 and 2, if the nonzero digit is $3, 5,$ or $7$, there is only one choice of starting digit that is $2$-forceable. We have already seen that if the leading digit is $3, 5,$ or $7$, we can force two zeros with a greedy algorithm, and if the leading digit is $1$ or $9$, we can force three zeros with a greedy algorithm. Then $\ex[\text{\# of zeros between nonzero digits}]$ is equal to \[\dfrac{3}{4}\ex[\text{Starting digit 3, 5, or 7}] + \dfrac{1}{4}\ex[\text{Starting digit 1 or 9}].\]

For convenience, let $N = \ex[\text{Starting digit 3, 5, or 7}]$ and $M = \ex[\text{Starting digit 1 or 9}]$.
If we assume that the parity of each digit thereafter is uniformly random, then the probability that we obtain exactly $n$ zeros after the two forced zeros is $\frac{1}{2^{n+1}}$. Hence, \[N = 2 + \sum_{n=1}^{\infty} \dfrac{n}{2^{n+1}}.\]

From elementary calculus, we have \[\sum_{n=0}^{\infty} nr^n = \dfrac{r}{(1-r)^2}.\] Therefore,  $N$ evaluates as \begin{equation}2 + \sum_{n=1}^{\infty} \dfrac{n}{2^{n+1}} = 2 + \dfrac{1}{2}\left(\dfrac{\frac{1}{2}}{(1-\frac{1}{2})^2}\right) = 3.\end{equation}

We now turn to evaluating $M$. Intuitively, Corollary \ref{cor:4} tells us that $M$ should be larger than $N$.
More precisely, Corollary \ref{cor:4} shows that we have three guaranteed zeros. Assuming that every zero thereafter is uniformly random, we have \[M = 3 + \sum_{n=1}^{\infty} \dfrac{n}{2^{n+1}} = 4\]
\vspace{0.25cm}

Thus, $\ex[\text{\# of zeros between nonzero digits}] = \dfrac{1}{4} M + \dfrac{3}{4}N = \dfrac{1}{4} \cdot 4 + \dfrac{3}{4} \cdot 3 = \dfrac{13}{4}$, as desired.
\end{proof}

The estimation of $\frac{13}{4} = 3.25$ is remarkably close to the empirical average of $3.247$ gathered for $d_{1013}$ (see Tables 1 and 2). 

\section{Acknowledgments}
The author would like to thank Lawrence C. Washington for suggesting the problem and for his helpful guidance throughout the process. 
\eject

\bibliography{bib}
\bibliographystyle{ieeetr}

\eject

\begin{appendices}
\section{Data}
\begin{table}[!htb]
\label{first-q}
\begin{tabular}{r|r|r} \hline
\multicolumn{1}{c|}{Gap Size} & \multicolumn{1}{c|}{Frequency} & \multicolumn{1}{c}{Weighted Sum} \\ \hline
0                             & 0                              & 0                                \\
1                             & 0                              & 0                                \\
2                             & 98                             & 196                              \\
3                             & 81                             & 243                              \\
4                             & 34                             & 136                              \\
5                             & 18                             & 90                               \\
6                             & 11                             & 66                               \\
7                             & 5                              & 35                               \\
8                             & 2                              & 16                               \\
9                             & 1                              & 9                                \\
10                            & 1                              & 10                               \\
11                            & 0                              & 0                                \\
12                            & 1                              & 12                               \\ \hline
\multicolumn{1}{l|}{Average}  & 3.226190476                    & \multicolumn{1}{l}{}         
\end{tabular}
\quad
\begin{tabular}{r|r|r} \hline
\multicolumn{1}{c|}{Gap Size} & \multicolumn{1}{c|}{Frequency} & \multicolumn{1}{c}{Weighted Sum} \\ \hline
0                             & 0                              & 0                                \\
1                             & 0                              & 0                                \\
2                             & 193                            & 386                              \\
3                             & 159                            & 477                              \\
4                             & 69                             & 276                              \\
5                             & 44                             & 220                              \\
6                             & 19                             & 114                              \\
7                             & 15                             & 105                              \\
8                             & 4                              & 32                               \\
9                             & 1                              & 9                                \\
10                            & 1                              & 10                               \\
11                            & 0                              & 0                                \\
12                            & 1                              & 12                               \\ \hline
\multicolumn{1}{l|}{Average}  & 3.243083004                    & \multicolumn{1}{l}{}            
\end{tabular}
\caption{First and second quartiles for frequency versus number of zeros. The third column gives a more accurate sense of how much each data point contributes to the average gap. The data were collected by running the algorithm until 1013 nonzero digits were found. That is, $d_{1013}$ was computed.}
\end{table}

\begin{table}[!htb]
\begin{tabular}{r|r|r} \hline
\multicolumn{1}{c|}{Gap Size} & \multicolumn{1}{c|}{Frequency} & \multicolumn{1}{c}{Weighted Sum} \\ \hline
0                             & 0                              & 0                                \\
1                             & 0                              & 0                                \\
2                             & 293                            & 586                              \\
3                             & 235                            & 705                              \\
4                             & 112                            & 448                              \\
5                             & 62                             & 310                              \\
6                             & 26                             & 156                              \\
7                             & 21                             & 147                              \\
8                             & 7                              & 56                               \\
9                             & 1                              & 9                                \\
10                            & 1                              & 10                               \\
11                            & 1                              & 11                               \\
12                            & 1                              & 12                               \\ \hline
\multicolumn{1}{l|}{Average}  & 3.223684211                    & \multicolumn{1}{l}{}            
\end{tabular}
\quad
\begin{tabular}{r|r|r} \hline
\multicolumn{1}{c|}{Gap Size} & \multicolumn{1}{c|}{Frequency} & \multicolumn{1}{c}{Weighted Sum} \\ \hline
0                             & 0                              & 0                                \\
1                             & 0                              & 0                                \\
2                             & 404                            & 808                              \\
3                             & 294                            & 882                              \\
4                             & 145                            & 580                              \\
5                             & 86                             & 430                              \\
6                             & 41                             & 246                              \\
7                             & 26                             & 182                              \\
8                             & 9                              & 72                               \\
9                             & 4                              & 36                               \\
10                            & 1                              & 10                               \\
11                            & 2                              & 22                               \\
12                            & 2                              & 24                               \\ \hline
\multicolumn{1}{l|}{Average}  & 3.246548323                    & \multicolumn{1}{l}{}            
\end{tabular}
\caption{Third and fourth quartiles for frequency versus number of zeros. The third column gives a more accurate sense of how much each data point contributes to the average gap.}
\end{table}

\eject
\begin{table}[!htb]
\centering
\label{freq}
\begin{tabular}{r|r|r|r|r|r|l} \hline 
\multicolumn{1}{c|}{Digit} & \multicolumn{1}{c|}{1} & \multicolumn{1}{c|}{3} & \multicolumn{1}{c|}{5} & \multicolumn{1}{c|}{7} & \multicolumn{1}{c|}{9} & Total                     \\ \hline
Frequency                  & 55                     & 128                    & 139                    & 124                    & 67                     & \multicolumn{1}{r}{513}   \\ \cline{2-6}
                           & 0.107                  & 0.250                  & 0.271                  & 0.242                  & 0.131                  &                           \\ \hline
\multicolumn{1}{c|}{Digit} & \multicolumn{1}{c|}{1} & \multicolumn{1}{c|}{3} & \multicolumn{1}{c|}{5} & \multicolumn{1}{c|}{7} & \multicolumn{1}{c|}{9} & \multicolumn{1}{c}{Total} \\ \hline
Frequency                  & 113                    & 260                    & 272                    & 244                    & 124                    & \multicolumn{1}{r}{1013}  \\ \cline{2-6}
                           & 0.112                  & 0.257                  & 0.269                  & 0.241                  & 0.122                  &                          
\end{tabular}

\caption{Frequency of nonzero digits. Note the higher than expected frequency of 5 and the lower than expected frequencies of 1 and 9.}
\end{table}

\begin{table}[!htb]
\centering

\label{my-dasdf}
\begin{tabular}{l|rrrrrrrrrrr}
\hline
\multicolumn{1}{c|}{\textbf{}} & \multicolumn{11}{c}{\textbf{Number of zeros}}                                                                                                                                                                                                                                                \\ \hline
\multicolumn{1}{c|}{Digit}     & \multicolumn{1}{c|}{2}   & \multicolumn{1}{c|}{3}  & \multicolumn{1}{c|}{4}  & \multicolumn{1}{c|}{5}  & \multicolumn{1}{c|}{6}  & \multicolumn{1}{c|}{7}  & \multicolumn{1}{c|}{8} & \multicolumn{1}{c|}{9} & \multicolumn{1}{c|}{10} & \multicolumn{1}{c|}{11} & \multicolumn{1}{c}{12} \\
1                              & \multicolumn{1}{r|}{0}   & \multicolumn{1}{r|}{46} & \multicolumn{1}{r|}{35} & \multicolumn{1}{r|}{19} & \multicolumn{1}{r|}{8}  & \multicolumn{1}{r|}{2}  & \multicolumn{1}{r|}{1} & \multicolumn{1}{r|}{2} & \multicolumn{1}{r|}{0}  & \multicolumn{1}{r|}{0}  & 0                      \\
3                              & \multicolumn{1}{r|}{140} & \multicolumn{1}{r|}{72} & \multicolumn{1}{r|}{23} & \multicolumn{1}{r|}{9}  & \multicolumn{1}{r|}{6}  & \multicolumn{1}{r|}{2}  & \multicolumn{1}{r|}{3} & \multicolumn{1}{r|}{0} & \multicolumn{1}{r|}{1}  & \multicolumn{1}{r|}{2}  & 1                      \\
5                              & \multicolumn{1}{r|}{133} & \multicolumn{1}{r|}{66} & \multicolumn{1}{r|}{24} & \multicolumn{1}{r|}{23} & \multicolumn{1}{r|}{11} & \multicolumn{1}{r|}{11} & \multicolumn{1}{r|}{3} & \multicolumn{1}{r|}{0} & \multicolumn{1}{r|}{0}  & \multicolumn{1}{r|}{0}  & 1                      \\
7                              & \multicolumn{1}{r|}{129} & \multicolumn{1}{r|}{54} & \multicolumn{1}{r|}{27} & \multicolumn{1}{r|}{18} & \multicolumn{1}{r|}{5}  & \multicolumn{1}{r|}{8}  & \multicolumn{1}{r|}{1} & \multicolumn{1}{r|}{2} & \multicolumn{1}{r|}{0}  & \multicolumn{1}{r|}{0}  & 0                      \\
9                              & \multicolumn{1}{r|}{0}   & \multicolumn{1}{r|}{56} & \multicolumn{1}{r|}{36} & \multicolumn{1}{r|}{17} & \multicolumn{1}{r|}{11} & \multicolumn{1}{r|}{3}  & \multicolumn{1}{r|}{1} & \multicolumn{1}{r|}{0} & \multicolumn{1}{r|}{0}  & \multicolumn{1}{r|}{0}  & 0                      \\
\end{tabular}
\caption{The number of zeros between consecutive nonzero digits for 1013 nonzero digits.}
\end{table}

\begin{table}[!htb]
\centering
\label{my-freq}
\begin{tabular}{l|rrrrrrrrrrr}
\hline
\multicolumn{1}{c|}{\textbf{}} & \multicolumn{11}{c}{\textbf{Probabilities}}                                                                                                                                                                                                                              \\ \hline
Digit                          & \multicolumn{1}{c}{2} & \multicolumn{1}{c}{3} & \multicolumn{1}{c}{4} & \multicolumn{1}{c}{5} & \multicolumn{1}{c}{6} & \multicolumn{1}{c}{7} & \multicolumn{1}{c}{8} & \multicolumn{1}{c}{9} & \multicolumn{1}{c}{10} & \multicolumn{1}{c}{11} & \multicolumn{1}{c}{12} \\
1                              & 0.000                 & 0.407                 & 0.310                 & 0.168                 & 0.071                 & 0.018                 & 0.009                 & 0.018                 & 0.000                  & 0.000                  & 0.000                  \\
3                              & 0.541                 & 0.278                 & 0.089                 & 0.035                 & 0.023                 & 0.008                 & 0.012                 & 0.000                 & 0.004                  & 0.008                  & 0.004                  \\
5                              & 0.489                 & 0.243                 & 0.088                 & 0.085                 & 0.040                 & 0.040                 & 0.011                 & 0.000                 & 0.000                  & 0.000                  & 0.004                  \\
7                              & 0.529                 & 0.221                 & 0.111                 & 0.074                 & 0.020                 & 0.033                 & 0.004                 & 0.008                 & 0.000                  & 0.000                  & 0.000                  \\
9                              & 0.000                 & 0.452                 & 0.290                 & 0.137                 & 0.089                 & 0.024                 & 0.008                 & 0.000                 & 0.000                  & 0.000                  & 0.000                  \\ 
\end{tabular}

\caption{The corresponding probabilities for gaps of size $n$ following the nonzero digits, computed directly from Table \ref{my-freq}.}
\end{table}
\eject

\section{Code}

We define our algorithm as follows, where $m = p_k$.

\begin{algorithm}
\caption{Selecting Powers of Two}
\label{CHalgorithm}
\begin{algorithmic}[1]
\Procedure{findzeros}{$m$, $k$}
\State $s  \leftarrow 2^{m} \pmod{10^k}$ 
\State $f \leftarrow 2^{4 \cdot 5^{k-2}} \pmod{10^{k}}$
\For{$i = 1, 2, ..., 5$}
\State $x \leftarrow s \cdot f \pmod{10^k}$
\If{the leading digit of $x$ is zero} 
	\State $m \leftarrow m + k \cdot 4 \cdot 5^{k-2}$
	\State \Return findzeros($m$, $k+1$)
    \EndIf
\EndFor
\If{we cannot force a zero} \Return{($m$, $k-1$)} \EndIf
\EndProcedure
\Procedure{nextTerm}{$c$, $e$}
\State maxCurrent $\leftarrow c$
\State maxExponent $\leftarrow e$
\For {$n = 1, 2, ..., 5$}
    \State $c \leftarrow c + 4 \cdot 5^{e-1}$ 
    \State $t \leftarrow $ findzeros($c$, $e+2$)
    \If {$t[2] > $ maxExponent}
    	\State maxExponent $\leftarrow t[2]$
        \State maxCurrent $\leftarrow t[1]$
    \EndIf
\EndFor
\State \Return (maxCurrent, maxExponent)
\EndProcedure
\end{algorithmic}
\end{algorithm}

\end{appendices}

\end{document}